\documentclass[draft,reqno]{amsart}

\usepackage[Symbol]{upgreek}

\usepackage{amssymb}
\usepackage{epic}
\usepackage{mathrsfs}
\usepackage{accents}
\usepackage{stmaryrd}

\newcommand{\bbold}{\mathbb}

\def\R { {\bbold R} }
\def\Q { {\bbold Q} }

\def\N { {\bbold N} }
\def\T { {\bbold T} }

\def \I{\operatorname{I}}
\def \J{\bbold{J}}

\def \ex{\operatorname{e}}

\renewcommand\epsilon{\varepsilon}

\def \<{\langle}
\def \>{\rangle}

\def \supp {\operatorname{supp}}

\def \((  {(\!(}
\def \)) {)\!)}

\def \nate{\operatorname{e}}

\DeclareMathSymbol{\precequ}{\mathrel}{symbols}{"16}
\DeclareMathSymbol{\succequ}{\mathrel}{symbols}{"17}

\newtheorem{theorem}{Theorem}[section]
\newtheorem{lemma}[theorem]{Lemma}
\newtheorem{prop}[theorem]{Proposition}
\newtheorem{cor}[theorem]{Corollary}
\newtheorem{theoremintro}{Theorem}

\theoremstyle{definition}

\theoremstyle{remark}

\newtheorem*{examples}{Examples}

\def \card {\operatorname{card}}

\def \fA {{\mathfrak A}}
\def \fM {{\mathfrak M}}

\def \fm {{\mathfrak m}}

\def \No{\text{{\bf No}}}
\def \On{\text{{\bf On}}}
\def \BM{\operatorname{BM}}

\let\oldi\i
\let\oldj\j
\renewcommand\i{\relax\ifmmode{\boldsymbol{i}}\else\oldi\fi}
\renewcommand\j{\relax\ifmmode{\boldsymbol{j}}\else\oldj\fi}

\renewcommand\leq{\leqslant}
\renewcommand\geq{\geqslant}
\renewcommand\preceq{\preccurlyeq}

\renewcommand\le{\leq}
\renewcommand\ge{\geq}
\renewcommand\frak{\mathfrak}

\DeclareMathAlphabet{\mathbf}{OML}{cmm}{b}{it}

\DeclareFontFamily{U}{fsy}{}
\DeclareFontShape{U}{fsy}{m}{n}{<->s*[.9]psyr}{}
\DeclareSymbolFont{der@m}{U}{fsy}{m}{n}
\DeclareMathSymbol{\der}{\mathord}{der@m}{182}

\DeclareSymbolFont{der@m}{U}{fsy}{m}{n}
\DeclareMathSymbol{\derdelta}{\mathord}{der@m}{100}

\DeclareSymbolFont{imag@m}{OT1}{cmr}{m}{ui}
\DeclareMathSymbol{\imag}{\mathord}{imag@m}{105}

\DeclareFontFamily{OMS}{smallo}{}
\DeclareFontShape{OMS}{smallo}{m}{n}{<->s*[.65]cmsy10}{}
\DeclareSymbolFont{smallo@m}{OMS}{smallo}{m}{n}
\DeclareMathSymbol{\smallo}{\mathord}{smallo@m}{79}

\DeclareFontFamily{OMS}{largerdot}{}
\DeclareFontShape{OMS}{largerdot}{m}{n}{<->s*[.8]cmsy10}{}
\DeclareSymbolFont{largerdot@m}{OMS}{largerdot}{m}{n}
\DeclareMathSymbol{\largerdot}{\mathord}{largerdot@m}{15}

\DeclareMathSymbol{\llambda}{\mathord}{der@m}{108}
\DeclareMathSymbol{\rrho}{\mathord}{der@m}{114}

\def \Upl{\Uplambda}
\def \upo{\upomega}
\def \Upo{\Upomega}

\newcommand{\equationqed}[1]{\[\pushQED{\qed}#1 \qedhere\popQED\]\let\qed\relax}
\newcommand{\alignqed}[1]{\begin{align*}\pushQED{\qed} #1 \qedhere\popQED\end{align*}\let\qed\relax}

\makeatletter
\newcommand{\dminus}{\mathbin{\text{\@dminus}}}

\newcommand{\@dminus}{%
  \ooalign{\hidewidth\raise1ex\hbox{\bf.}\hidewidth\cr$\m@th-$\cr}%
}
\makeatother

\begin{document}

\title{The Surreal Numbers as a Universal $H$-field}

\author[Aschenbrenner]{Matthias Aschenbrenner}
\address{Department of Mathematics\\
University of California, Los Angeles\\
Los Angeles, CA 90095\\
U.S.A.}
\email{matthias@math.ucla.edu}

\author[van den Dries]{Lou van den Dries}
\address{Department of Mathematics\\
University of Illinois at Urbana-Cham\-paign\\
Urbana, IL 61801\\
U.S.A.}
\email{vddries@math.uiuc.edu}

\author[van der Hoeven]{Joris van der Hoeven}
\address{\'Ecole Polytechnique\\
91128 Palaiseau Cedex\\
France}
\email{vdhoeven@lix.polytechnique.fr}

\begin{abstract} We show that the natural embedding of the differential field
of transseries into Conway's field of surreal numbers with the  
Berarducci-Mantova derivation is an elementary embedding. We also prove 
that any Hardy field embeds into the field of surreals with the Berarducci-Mantova derivation. 
\end{abstract}

\date{August 2016}

\maketitle

\section*{Introduction}

\noindent
Berarducci and Mantova~\cite[Theorem B]{BM} have recently 
constructed a derivation (denoted
by $\der_{\BM}$ below) on Conway's ordered field $\No$ of surreal numbers that makes the latter a Liouville closed $H$-field with constant field~$\R$. The standard example of such an object is the ordered differential field~$\T$ of transseries, and the question arises whether $\No$ with $\der_{\BM}$ is elementarily equivalent to $\T$. Below we give a positive answer in a stronger form: Theorem~\ref{BM1}.
Throughout this paper we consider $\No$ as a
differential field with derivation $\der_{\BM}$.

Both $\No$ and $\T$ are also exponential fields; the exponential function~$\exp$ on $\No$ is
defined in Gonshor~\cite{G}. We refer to \cite[Appendix A]{ADH}
for the precise construction of $\T$, but the ``generating element'' $x$ of $\T$ there will be denoted by $x_{\T}$ here, since we prefer to have $x$ range here over arbitrary surreal numbers. It is folklore (but see Section~\ref{sec:emb}
for a proof) 
that there is a unique embedding $\iota\colon \T\to \No$
of ordered exponential fields with $\iota(x_{\T})=\omega$ that is the identity on $\R$ and respects infinite sums. It follows easily from Wilkie's theorem
~\cite{W} and other known facts that $\iota$ is an elementary embedding of ordered exponential fields; see Section~\ref{sec:emb} for details. 
The analogue for the derivation instead of the exponentiation 
requires more effort:

\begin{theoremintro}\label{BM1} The mapping $\iota\colon \T \to \No$ is an elementary embedding of ordered differential fields.
\end{theoremintro}

\noindent
This answers a question posed in \cite{BM}. The main tools
for proving this result come from \cite[Theorems~15.0.1 and~16.0.1]{ADH}. These tools enable us to reduce the proof
of Theorem~\ref{BM1} to exhibiting 
$\No$ as a directed union
of subfields $\R[[\omega^\Gamma]]$ that are closed under 
$\der_{\BM}$ and where $\Gamma$ is an ordered additive subgroup of~$\No$ having a smallest nontrivial archimedean class; exhibiting $\No$ as 
such a directed union makes up an important part of our paper.
(As a byproduct we get a new proof that $\der_{\BM}(\No)=\No$.) We use the same kind of reduction to obtain:
  
\begin{theoremintro}\label{BM2} The surreals of countable length form a subfield of $\No$ closed under~$\der_{\BM}$. As a differential subfield of $\No$ it is an elementary submodel of~$\No$.
\end{theoremintro}

\noindent
This also uses a result of Esterle~\cite{E} and its consequence that 
for any countable ordinal $\alpha$, any well-ordered set of surreals of length
$<\alpha$ is countable: Lemma~\ref{c2}.  

Finally, we establish an embedding result for $H$-fields:

\begin{theoremintro}\label{BM3} Every $H$-field with small derivation and constant field~$\R$ can be embedded over~$\R$ as an ordered differential field into 
$\No$.
\end{theoremintro}

\noindent
Thus every Hardy field extending $\R$ embeds over $\R$ as an ordered differential field into $\No$. Despite these excellent properties of $\der_{\BM}$, 
Schmeling's thesis~\cite{S} gives us reason to believe that $\der_{\BM}$ is not yet the ``best'' derivation on $\No$. We expect to address this issue in 
later papers.

\medskip\noindent
We thank Philip Ehrlich and Elliot Kaplan for giving us useful information about initial substructures of $\No$ of various kinds. We also thank the referee 
for pointing out places where more detail was needed and for debunking our 
initial attempt to prove Lemma~\ref{c2}.

\section{Preliminaries}\label{sec:pre} 

\noindent
Here we fix notation and terminology and summarize the results from \cite{ADH, BM, G} that we need as background material and
as tools in our proofs.

\subsection*{Notations and terminology}
Below, $m$,~$n$ range over $\N=\{0,1,2,\dots\}$, and 
$\alpha$,~$\beta$ and $\mu$,~$\nu$ range over ordinals. (The letter $\lambda$ will serve another purpose, as in \cite{BM}.)

As in \cite{G}, a {\em surreal number\/} is by definition a function $a\colon \mu \to \{{-},{+}\}$ on an ordinal $\mu=\{\alpha:\ \alpha< \mu\}$. For such $a$ we let
$l(a):=\mu$ be the {\em length}\/ of $a$. From now on we let $a$,~$b$,~$x$,~$y$ be surreal numbers. The class~$\No$ of surreal
numbers carries a canonical linear ordering $<$: $a< b$ iff $a$ is lexicographically less than $b$, where by convention we set $a(\mu):=0$ for~${\mu\ge l(a)}$ and linearly order~$\{{-},0,{+}\}$ by ${-} < 0 < {+}$.
We also have the canonical partial ordering~$<_s$ on~$\No$ given by: $a<_s b$ (``$a$ is simpler than~$b$'') iff $a$ is a proper initial segment of $b$, that is, $l(a) < l(b)$, and $a|_{\mu} =b|_{\mu}$ for $\mu:=l(a)$. 
For sets $A,B\subseteq \No$ with $A<B$ (that is, $a<b$ for all $a\in A$ and $b\in B$) we
let $x=A|B$ mean that $x$ is the simplest surreal with $A<x<B$, as in \cite{G} and \cite{BM}. We also use the terms ``canonical representation'' and ``monomial representation'' (of a surreal number) as in \cite{BM}.

 The ordinal $\alpha$ is identified with the surreal $a\colon \alpha \to \{ {-},{+}\}$ with $a(\beta)={+}$ for all~$\beta<\alpha$. A useful fact is the equivalence $\alpha < x\Longleftrightarrow \alpha\dot{+}1\le_s x$, where~$\alpha\dot{+}1$ is the successor ordinal to $\alpha$. The subclass of $\No$ consisting of the ordinals is denoted by~$\On$. A set $S\subseteq \No$ is said to be {\em initial\/} if $x\in S$ whenever 
$x<_s y\in S$. As in \cite{DE} we set $\No(\alpha)=\big\{x:\ l(x)<\alpha\big\}$, an initial subset of $\No$. 

 We refer to \cite{G} or \cite{BM} for the inductive definitions of the binary operations of addition and multiplication on $\No$ that make $\No$ into a real closed field, with the ordinal $0$ as its zero element and the ordinal
$1$ as its multiplicative identity. The field ordering of this real closed field is the above lexicographic linear ordering $<$. This field $\No$ contains $\R$ as an initial subfield in the way specified in \cite{G}. The field sum $\alpha+n$ equals the ordinal sum $\alpha\dot{+}n$.
Each initial set $\No(\omega^\alpha)$ underlies an additive subgroup of $\No$; see \cite{DE}.

  Let $\Gamma$ be an (additively written) ordered abelian group. Then we set 
$$\Gamma^{>}\ :=\ \{\gamma\in \Gamma:\ \gamma>0\}.$$ We use this notation also for the underlying additive groups of $\No$ and $\R$, so $\No^{>}=\{a:\ a>0\}$, and $\R^{>}:=\{r\in \R:\ r>0\}$. For $\gamma\in \Gamma$ we define
$$[\gamma]\ :=\ \big\{\delta\in \Gamma:\ \text{$|\delta|\le n|\gamma|$ and $|\gamma|\le n|\delta|$ for some $n\ge 1$}\big\},$$
the {\em archimedean class of $\gamma$\/} (in $\Gamma$). The archimedean classes
of elements of $\Gamma$ partition the set $\Gamma$, and we totally order this set of archimedean classes by
$$[\gamma_1]\ <\ [\gamma_2]\quad :\Longleftrightarrow\quad  \text{$n|\gamma_1|\ <\ |\gamma_2|$ for all $n\ge 1$.}$$ 
Thus the least archimedean class is $[0]=\{0\}$, 
the {\em trivial\/} archimedean class. 

The convex hull of $\R$ in $\No$ is a valuation ring $V$ of the field $\No$. We consider~$\No$ accordingly as a {\em valued\/} field whose (Krull) valuation $v$ has $V$ as its valuation ring. For any (Krull) valued field
$K$ with valuation $v$ and elements $f,g\in K$ we let $f\preceq g$, $f\prec g$, $f\asymp g$, $f\sim g$ abbreviate $v(f)\ge v(g)$, $v(f) > v(g)$, $v(f)=v(g)$, and
$v(f-g)>vf$. (See \cite[Section~3.1]{ADH}.) We shall use these notations in particular
for the valued field $\No$.  

\subsection*{The omega map, the Conway normal form, and summability} We assume familiarity with Conway's omega map $x\mapsto \omega^x\colon \No\to \No^{>}$. Recall that $\omega^x$ is the simplest positive element in its archimedean class; so $\omega^x\prec \omega^y$ whenever~$x<y$. See~\cite{G} for details, including the proof that each~$a$ has a unique representation 
$$a\ =\ \sum_x a_x\omega^x \qquad(\text{the Conway normal form of } a)$$ with real coefficients $a_x$ such that $E(a):=\{x:\ a_x\ne 0\}$ is a subset of $\No$ (not just a subclass) and is reverse well-ordered. This will be the meaning of $E(a)$ and $a_x$ throughout. The {\em leading monomial of $a$\/} is
$\omega^x$ with $x=\max E(a)$, for $a\ne 0$. 
The {\em terms\/} of $a$ are the $a_x\omega^x$ with $a_x\ne 0$. 
The omega map extends the
usual ordinal exponentiation $\alpha\mapsto \omega^\alpha$.
Given any set $S\subseteq \No$ we let $\R[[\omega^S]]$ denote the additive subgroup of $\No$ consisting of the surreals $a$ with $E(a)\subseteq S$.  

\medskip\noindent
Let $(a_i)_{i\in I}$ be a family of surreals; this includes
$I$ being a set.
We say that $(a_i)$ is {\em summable\/}
(or that $\sum_i a_i$ exists) if $\bigcup_i E(a_i)$ is reverse well-ordered, and for each~$x$ there are only finitely many $i\in I$ with $x\in E(a_i)$; in that case we set $\sum_i a_i:=\sum_x \left(\sum_i a_{i,x}\right)\omega^x$. If $S$ is a subset of $\No$, then for any summable family $(a_i)$ in~$\R[[\omega^S]]$ we have $\sum_i a_i\in \R[[\omega^S]]$.

\medskip\noindent
As in \cite{BM}, we let $\fM$ denote the class of {\em monomials\/}
$\omega^x$; so $\fM$ is a multiplicative subgroup of $\No^\times$.
The Conway normal form allows us to consider any surreal number $a$
as a {\em generalized series}
$$ a\ =\ \sum_{\fm \in \fM} a_\fm \fm $$
with coefficients $a_\fm\in \R$, monomials  $\fm\in \fM$, and reverse
well-ordered {\em support\/} $\supp a := \{ \fm \in \fM: a_\fm \neq 0  \} = \omega^{E(a)}$.
This makes the above notion of summability for surreal numbers coincide
with the corresponding notion for generalized series from~\cite[Section~1.5]{S}.

Next, $\J:=\big\{a:\ E(a)\subseteq\No^{>}\big\}$ is the class of {\em purely
infinite\/} surreals, an additive subgroup of $\No$ that is moreover closed under multiplication. 
Thus $\fM\cap \J=\fM^{\succ 1}$, and $\No = \J \oplus \R \oplus \No^{\prec 1}$. 

\subsection*{Exponentiation, and the functions $g$ and $h$}
Gonshor~\cite{G} gave an inductive definition of the exponential function $\exp\colon \No \to \No^{>}$, and established its basic properties. These include $\exp$ being an order-preserving isomorphism from the additive group of $\No$ onto its multiplicative group of positive elements. The inverse of
$\exp$ is of course denoted by $\log\colon \No^{>}\to \No$.
The $n$th iterate of the map $\exp\colon \No \to \No$ is denoted by
$\exp_n$, so $\exp_0$ is the identity map on $\No$, and
$\exp_1 (x)=\exp(x)$. Also $\nate^x:= \exp(x)$. The logarithmic map $\log$ maps
$\No^{>\N}$ into itself; the $n$th iterate of the
restriction of $\log$ to a map $\No^{>\N}\to \No^{>\N}$ is
denoted by $\log_n$, so $\log_0$ is the identity map on
$\No^{>\N}$ and $\log_1(x)=\log(x)$ for $x>\N$. 

The exponential map $\exp$ and the omega-map $x\mapsto \omega^x$ are related by the order preserving bijection 
$g\colon \No^{>} \to \No$, which satisfies 
 $$\exp(\omega^x)\ =\ \omega^{\omega^{g(x)}}\quad \text{ for all $x>0$.}$$
We have $g(n)=n$ for all $n$. More generally, Theorem~10.14 in~\cite{G} says that
$g(\alpha)=\alpha$ unless 
$\epsilon\le \alpha < \epsilon + \omega$ for some 
$\epsilon$-number, in which case $g(\alpha)=\alpha+1$. (An {\it $\epsilon$-number}\/ is an ordinal $\epsilon$ such that $\omega^\epsilon=\epsilon$.) We shall need $g(x)$ mainly in the other extreme case where
$x$ has the form $\omega^{-\alpha}$. Here Theorem~10.15 in~\cite{G} gives $g(\omega^{-\alpha})=-\alpha+1$. 

\medskip\noindent
We also use the inverse $h\colon \No \to  \No^{>}$ of $g$.
Note that
$$\omega^{\omega^y}\ =\ \exp(\omega^{h(y)})\quad 
\text{ for all }y.$$
The result above for $g(\omega^{-\alpha})$ yields 
$h(-\alpha+1)= \omega^{-\alpha}$, from which we get
$$ \log \omega^{\omega^{-\alpha+1}}\ =\ \omega^{\omega^{-\alpha}}.$$
Applying this to the ordinal $\alpha+1$ instead of $\alpha$ we get 
$$ \log \omega^{\omega^{-\alpha}}\ =\ 
\omega^{\omega^{-(\alpha +1)}}.$$

\medskip\noindent
From \cite{G} we have $\exp(\J)=\fM$. Thus besides the Conway normal form and 
the series
representation, any surreal number $a$ also has
a unique representation
$$ a\ =\ \sum_{j \in \J} a_j\nate^j \qquad(\text{exponential normal form of } a) $$
with real coefficients $a_j$ and reverse well-ordered 
$\{ j \in \J: a_j \neq 0\}$; this is also called the {\em Ressayre form of $a$}.
For nonzero $a$ with leading monomial $\nate^b$,
$b\in \J$, we set $\ell(a):=b$. Then
$-\ell\colon \No^{\times} \to \J$ is a (Krull) valuation on the field~$\No$, and  
$$\big\{a: -\ell(a)\ge 0\big\}\ =\ \big\{a:\   \text{$|a|\le r$ for some $r\in \R^{\ge 0}$}\big\}\ =\ V,$$
so we may consider
$-\ell$ as the valuation of our valued field $\No$.     
Important in \cite{BM} is also the class 
$\fA$ of {\em log-atomic\/} surreals, consisting of the $a>\N$ all whose iterated logarithms $\log_n a$ lie in $\fM$. We have $\fA\subseteq \frak{M}^{\succ 1}$ and $\exp(\fA)=\log(\fA)=\fA$. It follows from $\fA\subseteq \fM$ that if $x,y\in \fA$ and $x<y$, then $x\prec y$. (In \cite{BM} the class of log-atomic surreals
is denoted by $\mathbb{L}$, but this notation conflicts with ours in other
papers.)

\subsection*{Surreal derivations.} We summarize here some results from \cite{BM} as needed, and add a few remarks. 
A {\em surreal derivation\/} is a derivation $\der$ on the field~$\No$ such that\begin{enumerate}
\item[(SD1)] $\big\{a: \der(a)=0\big\}=\R$;
\item[(SD2)] $\der(a)>0$ for all $a>\R$;
\item[(SD3)] $\der\big(\exp(a)\big)=\der(a)\exp(a)$ 
for all $a$;
\item[(SD4)] for any summable family $(a_i)$ of surreals, the
family $\big(\der(a_i)\big)$ is also summable, and $\der\left(\sum_i a_i\right)=\sum_i \der(a_i)$.
\end{enumerate}
The ordered field $\No$ equipped with any surreal derivation is an $H$-field; this doesn't need (SD3) or (SD4). The particular derivation $\der_{\BM}$ is surreal, maps $\fA$ into $\fM$, and is obtained in \cite{BM} as a special case of a rather general construction. Before we get to that, we mention 
Proposition~6.5 and Theorem~6.32 from that paper: \begin{enumerate}
\item[(BM1)] If $\der$ is a surreal derivation, then for all $x,y>\N$ with $x-y>\N$ we have 
$$\log \der(x)-\log \der(y)\ \prec\ x-y.$$
\item[(BM2)] Any map $D\colon\fA \to \R^{>}\fM$ such that for all $x,y\in \fA$,
$$D(\exp x)\ =\ D(x)\exp x, \quad 
\log D(x)-\log D(y)\ \prec\ \max(x,y),$$  
extends to a surreal derivation. 
\end{enumerate}
Thus (BM2) is a partial converse to (BM1), although the condition in (BM2)
that~$D$ takes only values in $\R^{>}\fM$ seems a rather severe restriction. 
We define a {\em pre-derivation\/} to be a map $D\colon\fA \to \R^{>}\fM$ as in (BM2). Note that if $D$ is a pre-derivation, then
\begin{equation}\label{eq:D(a)}\tag{$\ast$}
D(a)=\left(\prod_{m<n} \log_m a\right)\cdot D(\log_n a)\qquad\text{for all $a\in \fA$ and all $n$.}
\end{equation}
A pre-derivation $D$ actually extends canonically to a surreal
derivation $\der_D$. To define $\der_D$ in terms of $D$
we rely on the notion of {\em path derivatives},
introduced in~\cite{vdH:phd}, further developed in~\cite{S},
and adapted to the surreal setting in~\cite{BM}.
A {\em path\/} is a function $P\colon \N \to \R^{\times}\fM$
such that $P(n+1)$ is a term of $\ell(P(n))$, for all $n$. Given~$x$, the paths $P$ such that
$P(0)$ is a term of $x$ are the elements of a set $\mathcal{P}(x)$. For $x\in \fA$ there is a unique path 
$P\in \mathcal{P}(x)$; it is given by $P(n)=\log_n x$.
Thus if $P$ is a path and $P(m)\in \fA$, then $P(n)=\log_{n-m} P(m)$ for all $n\ge m$, so
$P(n)\in \fA$ for all $n\ge m$.

\medskip\noindent
Let $D$ be a pre-derivation. The {\em path derivative\/} $\der_D(P)\in \R\fM$ for a path $P$ is defined as follows,
with \eqref{eq:D(a)} guaranteeing independence of $n$ in (1): \begin{enumerate}
\item if $P(n)\in \fA$, then
$\der_D(P):= \left(\prod_{m<n} P(m)\right)\cdot D(P(n))$;
\item if $P(n)\notin \fA$ for all $n$, then $\der_D(P):= 0$.
\end{enumerate}
The rationale behind path derivatives is the following proposition:
\begin{enumerate}
\item[(BM3)] For each $a$ the family $\big(\der_D(P)\big)_{P\in \mathcal{P}(a)}$ is summable.
\end{enumerate}
This result is stated in \cite[Proposition~6.20]{BM} only for one particular
pre-derivation, but, as the authors mention, the proof extends to any pre-derivation.
In view of (BM3) we can now define $\der_D\colon \No \to \No$ by
$$\der_D(a)\ :=\ \sum_{P\in \mathcal{P}(a)} \der_D(P).$$
It follows from \eqref{eq:D(a)} that $\der_D$ extends $D$, and the arguments in \cite[Section~6]{BM} show that $\der_D$ is a surreal derivation.
 
\subsection*{Results from \cite{ADH}} To state the relevant facts, we 
recall from \cite{AD} or \cite{ADH} that an {\it $H$-field}\/ is by definition
an ordered differential field $K$ with derivation $\der$ and constant field~${C=\big\{f\in K:\ \der(f)=0\big\}}$ such that: \begin{enumerate}
\item[(H1)] $\der(f)>0$ for all $f\in K$ with $f>C$;
\item[(H2)] $\mathcal{O}=C+\smallo$, where $\mathcal{O}$ is the
convex hull of $C$ in $K$, and $\smallo$ is the maximal ideal
of the valuation ring $\mathcal{O}$.
\end{enumerate}
Let $K$ be an $H$-field, and let $\mathcal{O}$ and $\smallo$ be as in (H2). Thus $K$ is a valued field
with valuation ring $\mathcal{O}$. We consider $K$ in the natural way as an 
$\mathcal{L}$-structure, where 
$$\mathcal{L}:= \{\,0,\,1,\, {+},\, {-},\, {\times},\, \der,\, {\le},\, {\preceq}\,\}$$ is the language of ordered valued differential fields; in particular,
$$f\preceq g\ \Longleftrightarrow\ f\in \mathcal{O} g\ \Longleftrightarrow\ |f|\le c|g| \text{ for some $c\ge 0$ in }C.$$ 
Given $f\in K$ we also write $f'$ instead of $\der(f)$, and we set
$f^\dagger:= f'/f$ for $f\ne 0$, so
$(fg)^\dagger=f^\dagger + g^\dagger$ and $(1/f)^\dagger=-f^\dagger$ for $f,g\in K^\times$. A useful subset of the value group $\Gamma:= v(K^\times)$ of the valued field $K$ is
$$\Psi\ :=\ \Psi_K\ :=\ \big\{v(f^\dagger):\ f\in K^\times,\ f\not\asymp 1\big\}\ =\ \big\{v(f^\dagger):\ f\in K,\ f>C\big\}.$$
As in \cite{ADH} we call $K$ {\em grounded\/} if $\Psi$ has a largest element.
For the convenience of the reader we include a proof of the following wellknown fact.

\begin{lemma}\label{grarch} Assume $K$ has constant field $C=\R$. Then $K$ is grounded iff 
$\Gamma$ has a smallest nontrivial archimedean class. 
\end{lemma}
\begin{proof} Let $f,g\in K$, $f,g>C$. Suppose the archimedean class 
$\big[v(f)\big]=\big[v(1/f)\big]$ of
$v(f)$ is greater than $\big[v(g)\big]$. This means: $v(f) < nv(g)=v(g^n)<0$ for all 
$n\ge 1$. Hence $f^\dagger> (g^n)^\dagger=ng^\dagger>0$ for all $n\ge 1$, 
by~\cite[Lemma~1.4]{AD}, so $v(f^\dagger) < v(g^\dagger)$. A similar argument 
(which doesn't need $C=\R$) shows that if $\big[v(f)\big]=\big[v(g)\big]$, then 
$v(f^\dagger)=v(g^\dagger)$.  Thus we have an
order-reversing bijection $\big[v(f)\big] \mapsto v(f^\dagger)$ ($f\in K$, $f>C$)
from the set of nontrivial archimedean classes of $\Gamma$ onto $\Psi$.
\end{proof}

\noindent
An {\em $H$-subfield of~$K$\/} is by definition an ordered differential subfield of $K$ that is an $H$-field. 
In \cite{ADH} we axiomatized the elementary (=~first-order) theory of the $H$-field~$\T$
of transseries. This (complete) theory is called $T_{\text{small}}^{\text{nl}}$ there and its models are exactly the $H$-fields $K$ satisfying the following (first-order) conditions: \begin{enumerate}
\item the derivation of $K$ is small, that is, $\der\smallo\subseteq \smallo$;
\item $K$ is Liouville closed;
\item $K$ is $\upo$-free;
\item $K$ is newtonian.
\end{enumerate} 
(An $H$-field $K$ is said to be {\it Liouville closed}\/ if it is real closed and for all $f\in K$ there exists $g\in K$ with $g'=f$ and an $h\in K^\times$ such that $h^\dagger=f$; for the definition of ``$\upo$-free'' and ``newtonian'' we refer to the
Introduction of \cite{ADH}.)  
Dropping the smallness axiom (1), we get the incomplete but model complete
theory $T^{\text{nl}}$; see~\cite[Chapter~16]{ADH}. The 
$H$-field $\T$ satisfies (3) and (4) by ~\cite[Corollary~11.7.15 and Theorem~15.0.1]{ADH}, which for an arbitrary $H$-field $K$ amount to the following:

{\em If $\der K = K$ and $K$ is a directed union of spherically complete grounded $H$-sub\-fields, then $K$ is $\upo$-free and newtonian.}

The condition $\der K = K$ is automatically satisfied if
$K$ is a directed union of spherically complete grounded $H$-subfields $E$ such that for some $\phi\in E$ we have $v(\phi)=\max\Psi_E$ and $\phi\in \der K$, by \cite[Corollary 15.2.4]{ADH}.

\section{Infinite Products and Log-atomic Surreals}\label{sec:infprod}
 
\noindent
The pre-derivation $D$ in \cite{BM} with $\der_D=\der_{\BM}$ is
defined by a certain identity. Towards the end of this section we give this
identity a more suggestive form, which we found useful. 
But we begin with some remarks on $\epsilon$-numbers, which play an important role in the next sections.

\subsection*{Remarks on $\epsilon$-numbers}
Throughout this paper $\epsilon$ will denote an $\epsilon$-number, that is,
$\epsilon$ is an ordinal such that $\omega^\epsilon=\epsilon$.

\begin{lemma} For any $\alpha$ there is a least
$\epsilon$-number $\epsilon(\alpha)\ge \alpha$. Moreover, if $\alpha$ is infinite, then
$\card(\epsilon(\alpha))=\card(\alpha)$. 
\end{lemma}
\begin{proof} The recursion defining
$\omega^\alpha$ as a function of $\alpha$ easily yields
that this function is strictly increasing, with $\omega^\alpha\ge \alpha$, 
$\card(\omega^\alpha)= \max\big(\aleph_0,\card(\alpha)\big)$, and thus
$\card(\omega^\alpha)=\card(\alpha)$ if $\alpha$ is infinite.
Now define $\alpha_n$ as a function of $n$ by the recursion
$\alpha_0=\alpha$ and $\alpha_{n+1}=\omega^{\alpha_n}$.
Then $\sup_n \alpha_n$ is clearly the least
$\epsilon$-number $\ge \alpha$, and it has the same
cardinality as $\alpha$ if the latter is infinite.
\end{proof}

\noindent
If $\kappa$ is an uncountable cardinal, then by the remarks
in the proof above we have $\omega^\alpha < \kappa$ for all $\alpha < \kappa$. Thus uncountable cardinals are $\epsilon$-numbers. The least $\epsilon$-number is denoted by $\epsilon_0$, as usual, so $\epsilon_0=\sup_n \omega_n$ where the $\omega_n$ are defined by the recursion $\omega_0=\omega$ and $\omega_{n+1}=\omega^{\omega_n}$.

\subsection*{Infinite products of monomials} Recall that $\frak{M}$ is the multiplicative group of monomials 
$\omega^a$.  For a family
$(\frak{m}_i)$ in $\frak{M}$ we say that $\prod_i \frak{m}_i$
exists if $\sum_i a_i$ exists, with $\frak{m}_i=\omega^{a_i}$
for all $i$, and in that case, we set 
$$\prod_i \frak{m}_i\ :=\ \omega^{\sum_i a_i}\in \frak{M}.$$
The rules for manipulating these infinite products are easy consequences of those for infinite sums, and we shall freely use
them below. Note in particular that if~$(\frak{m}_i)$ is a family in $\frak{M}$ and $\prod_i \frak{m}_i$ exists, then
$\prod_i \frak{m}_i^{-1}$ exists and equals $(\prod_i \frak{m}_i)^{-1}$. 

In our definition of infinite products we could have represented monomials
as exponentials of elements in $\J$ instead of as powers of $\omega$.
Indeed, the equivalence between these options follows from
the next two lemmas:

\begin{lemma} Let $(a_i)$ be a summable family in $\mathbb{J}$. Then $\prod_i \exp(a_i)$ exists, and 
$$\exp\left(\sum_i a_i\right)\ =\ \prod_i \exp(a_i).$$
\end{lemma}
\begin{proof} We have $a_i=\sum_{x>0} a_{i,x}\omega^x$, so by \cite[Theorem 10.13]{G},
$$\exp(a_i)\ =\ \omega^{b_i}, \quad b_i\ :=\ \sum_{x>0} a_{i,x}\omega^{g(x)},$$
so $E(b_i)=g(E(a_i))$.
Since $\sum_i a_i$ exists, so does $\sum_i b_i$, and hence
$\prod_i \exp(a_i)=\prod_i \omega^{b_i}$ exists, and
$\prod_i \exp(a_i)=\omega^{\sum_i b_i}$. Moreover, with $\sum_i a_i = \sum_{x>0}a_x\omega^x$, we have $\sum_i b_i=\sum_{x>0} a_x \omega^{g(x)}$. Hence
again by \cite[Theorem 10.13]{G}, 
$$ \prod_i \exp(a_i)\ =\ \omega^{\sum_{x>0}a_x\omega^{g(x)}}\ =\  
\exp\left(\sum_{x>0} a_x\omega^x\right)\ =\ \exp\left(\sum_i a_i\right),$$
as claimed. 
\end{proof}

\begin{lemma} Let $(\fm_i)$ be a family in $\fM$ such that $\prod_i \fm_i$ exists. Then $\sum_i \log \fm_i$ exists, and
$\log \prod_i \fm_i = \sum_i \log \fm_i$.
\end{lemma}
\begin{proof} We have $\fm_i=\exp(a_i)$ with $a_i\in \J$, so
$a_i=\sum_{x>0} a_{i,x}\omega^x$, hence 
$$\fm_i\ =\ \omega^{b_i},\qquad b_i\ :=\ \sum_{x>0}a_{i,x}\omega^{g(x)}$$ 
by \cite[Theorem 10.13]{G}. Since the product $\prod_i \fm_i$
exists, so does $\sum_i b_i$, and therefore $\sum_i a_i=\sum_i \log \fm_i$ exists. Moreover, and again by \cite[Theorem 10.13]{G}, 
$$\prod_i \fm_i\ =\ \omega^{\sum_i b_i}\ =\ \omega^{\sum_{x>0} a_x \omega^{g(x)}}\ =\ \exp\left(\textstyle\sum\limits_{x>0}a_x\omega^x\right), \quad a_x:= \sum_i a_{i,x},$$
and so $\log \prod_i \fm_i=\sum_{x>0}a_x\omega^x=\sum_i a_i$.
\end{proof}

\subsection*{Log-atomic surreals} Recall that $\fA\subseteq \frak{M}^{\succ 1}$ is the class of log-atomic surreals. See \cite[Sections~1,~5]{BM} for the order-preserving bijection
$x\mapsto \lambda_x\colon \No \to \fA$ and for the fact that
$\lambda_x\le_s \lambda_y$ iff $x\le_s y$. It follows from
$\exp(\omega^x)=\omega^{\omega^{g(x)}}$ that $\fA\subseteq \omega^{\fM}$. Thus for any well-ordered index set $I$ and 
strictly decreasing map $i\mapsto \lambda_i\colon I \to \fA$
the product $\prod_i \lambda_i$ exists.  
We shall use Proposition~\ref{lambda2} and Corollary~\ref{lambda3} below to define the pre-derivation 
$\der_{\BM}|_{\fA}$.  

\begin{lemma}\label{lem:xplusone} Let $\frak{m}=A|B$ be a monomial representation with $\frak{m}\succ 1$. Then $$\exp(\frak{m})\ =\ \big(\frak{m}^{\N}\cup\exp(A)\big)\big|\exp(B).$$ 
\end{lemma}
\begin{proof} For $\frak{m}'< \frak{m}$ with
$\frak{m}'<_s \frak{m}$ we have $\frak{m'}\le a$ for some $a\in A$ (since $A< \frak{m}'< \frak{m}< B$ gives $\frak{m}\le_s \frak{m}'$). Likewise, for $\frak{m}< \frak{m}''<_s \frak{m}$, we have $b\le \frak{m''}$ for some $b\in B$.
It follows that for $\frak{m}'$  as above and $k\in \N^{\ge 1}$ we have $\exp(\frak{m}')^k\le \exp(a)$ for some $a\in A$, and that for $\frak{m}''$ as above and $k\in \N^{\ge 1}$ we have $\exp(b) \le \exp(\frak{m}'')^{1/k}$ for
some $b\in B$. This yields the desired result in view of
\cite[Theorem 3.8 (1)]{BM}. 
\end{proof}

\noindent
The monomial representation $\omega=\N|\emptyset$ shows that
in the conclusion of Lemma~\ref{lem:xplusone} we cannot drop 
$\frak{m}^{\N}$. Below we use the binary relations $\asymp^L$ and $\prec^L$ from \cite{BM}. Let $x=\{x'\}|\{x''\}$ be the canonical 
representation of $x$, and let $j,k$ range over $\N^{\ge 1}$. Then by \cite[Definition 5.12]{BM}, the defining representation of $\lambda_x$ is given by
$$\lambda_x\ =\ \left\{k, \exp_j\!\big(k\log_j(\lambda_{x'})\big) \right\}\big|\left\{\exp_j\!\big(\textstyle\frac{1}{k}\log_j(\lambda_{x''})\big)\right\}.$$
%where $x=\{x'\}|\{x''\}$ is the canonical 
%representation of $x$ and $j,k$ range over $\N^{\ge 1}$.  

\begin{prop}\label{xplusone} We have $\lambda_{x+1}=\exp(\lambda_x)$, and thus $\lambda_{x-1}=\log(\lambda_x)$.
\end{prop}
\begin{proof} Let $x=\{x'\}|\{x''\}$ be the canonical representation of $x$. Then $1=0|\emptyset$ gives $x+1=\{x, x'+1\}|\{x'' +1\}$.
Assume inductively that $\lambda_{x'+1}=\exp(\lambda_{x'})$
and $\lambda_{x''+1}=\exp(\lambda_{x''})$ for all $x'$ and $x''$. With $j$,~$k$ ranging over $\N^{\ge 1}$, \cite[5.15]{BM} gives
\begin{align*} \lambda_{x+1}\ &=\ \left\{k, \exp_j\!\big(k\log_j(\lambda_x)\big), 
\exp_j\!\big(k\log_j(\lambda_{x'+1})\big)\right\}\big|\left\{\exp_j\!\big(\textstyle\frac{1}{k}\log_j(\lambda_{x''+1})\big)\right\}\\
&=\ \left\{k, \exp_j\!\big(k\log_j(\lambda_x)\big), 
\exp_j\!\big(k\log_{j-1}(\lambda_{x'})\big)\right\}\big|\left\{\exp_j\!\big(\textstyle\frac{1}{k}\log_{j-1}(\lambda_{x''})\big)\right\}.
\end{align*}
The defining representation $\lambda_x=A|B$ is monomial, and the above gives
$\lambda_{x+1}=\N\cup S\cup\exp(A)|\exp(B)$ where $S$ includes
$\lambda_x^{\N}$ and all elements of $S$ are $\asymp^L \lambda_x$. Since $\lambda_x \prec^L \exp(\lambda_x)$, it follows easily from Lemma~\ref{lem:xplusone} that
$\lambda_{x+1}=\exp(\lambda_x)$. 
\end{proof}

\begin{lemma}\label{lambda2} We have $\lambda_{-\alpha}=\omega^{\omega^{-\alpha}}$.
\end{lemma}
\begin{proof} By induction on $\alpha$. The case $\alpha=0$ holds since $\lambda_0=\omega$. Assuming it holds for a certain
$\alpha$, we have 
$$\lambda_{-(\alpha+1)}\ =\  \log \lambda_{-\alpha}\ =\ \log \omega^{\omega^{-\alpha}}\ =\ \omega^{\omega^{-(\alpha +1)}}.$$
Next, let $\mu$ be an infinite limit ordinal. Then 
$-\mu=\emptyset|\{-\alpha:\ \alpha< \mu\}$, and so by \cite[5.15]{BM} and with $j$,~$k$ ranging over $\N^{\ge 1}$ we have 
$$\lambda_{-\mu}\ =\ \N\,\big|\left\{\exp_j\!\big(\textstyle\frac{1}{k}\log_j \lambda_{-\alpha}\big)\right\}.$$
Now  $\exp_j\!\big(\frac{1}{k}\log_j \lambda_{-\alpha}\big)\asymp^L \lambda_{-\alpha} \succ^L \lambda_{-\beta}$ when $\alpha < \beta$, so by cofinality and the inductive assumption we have 
$$\lambda_{-\mu}\ =\ 
\N\,\big|\big\{\omega^{\omega^{-\alpha}}:\ \alpha < \mu\big\}.$$
From $\N < \omega^{\omega^{-\mu}} < \omega^{\omega^{-\alpha}}$
for all $\alpha < \mu$, we get $\lambda_{-\mu}\le_s \omega^{\omega^{-\mu}}$. Take $a$ such that
$\lambda_{-\mu}=\omega^{\omega^{-a}}$. Then $\lambda_{-\mu} < \omega^{\omega^{-\alpha}}$ for $\alpha < \mu$ gives
$\omega^{-a} < \omega^{-\alpha}$ for all $\alpha < \mu$, and
thus $a> \alpha$ for all $\alpha < \mu$. This yields
$\mu \le_s a$, and thus $\omega^{\omega^{-\mu}}\le_s \lambda_{-\mu}$, hence $a=\mu$.
\end{proof}

\begin{lemma}\label{lambda2a} For $\lambda\in \fA$ we have:
$\lambda\ <\ \lambda_{-\alpha}\ \Longleftrightarrow\ 
\lambda_{-(\alpha+1)}\ \le_s\ \lambda$.
\end{lemma}
\begin{proof} For $\lambda=\lambda_x$ we have the equivalences
\alignqed{ \lambda_x\ <\ \lambda_{-\alpha}\ 
&\Longleftrightarrow\ x\ <\ -\alpha\ 
\Longleftrightarrow\ \alpha\ <\ -x\ \Longleftrightarrow\ 
\alpha+1\ \le_s\ -x\\ 
&\Longleftrightarrow\ -(\alpha+1)\ \le_s\ x\ 
\Longleftrightarrow\ \lambda_{-(\alpha+1)}\ \le_s\ \lambda_x.}
\end{proof} 

\subsection*{Transfinitely iterating the logarithm function}
In view of $\lambda_{-n}=\log_n\omega$ and the proof of Lemma~\ref{lambda2} it is suggestive to think of $\lambda_{-\alpha}$ as the $\alpha$ times iterated function $\log$ evaluated at $\omega$. Accordingly we set 
$\log_{\alpha}\omega:= \lambda_{-\alpha}$.
 We note that for $\beta < \alpha$ we have $-\beta <_s -\alpha$, so $\omega^{-\beta} <_s \omega^{-\alpha}$, and thus 
 $\log_{\beta}\omega <_s \log_{\alpha}\omega$. 
 
\begin{lemma}\label{lambda2b} Suppose $\alpha$ is an infinite limit ordinal. Then $\log_{\alpha} \omega$ is the simplest surreal $x>\N$ such that $x<\log_{\beta}\omega$ for all $\beta<\alpha$. 
\end{lemma}
\begin{proof} First, $\N<\log_{\alpha}\omega < \log_{\beta}\omega$ for all $\beta<\alpha$. Let $x$ be the simplest surreal $>\N$ such that $x< \log_{\beta}\omega$ for all $\beta<\alpha$.
Then $x$ is the simplest positive element in its archimedean class, so $x=\omega^y$ with $y>0$. Then $x=\omega^y < \omega^{\omega^{-\beta}}$ for $\beta< \alpha$ gives 
$y < \omega^{-\beta}$ for all $\beta< \alpha$. Then $y$ is the simplest positive element in its archimedean class:  if $0<y_0\le_s y$ and $y_0\le ny$, then $\omega^{y_0}\le_s \omega^y=x$ and $\N <\omega^{y_0}\le x^n < \log_{\beta}\omega$ for all $\beta<\alpha$, so $\omega^{y_0}=\omega^y$, and thus $y_0=y$. Hence $y=\omega^z$ with $z<-\beta$ for all $\beta<\alpha$, and thus $z\le -\alpha \le_s z$. Therefore, $\omega^{-\alpha}\le_s \omega^z=y$, so
$$ 
\log_{\alpha}\omega\ =\ \omega^{\omega^{-\alpha}}\ \le_s\ \omega^y\ =\ x,$$
and thus $\log_{\alpha}\omega = x$.
\end{proof}

\noindent
The surreals $\log_{\alpha}\omega$ occur in the definition of
$\der_{\BM}$ later in this section.

\subsection*{The $\kappa$-numbers}
The definition of $\der_{\BM}$ in \cite{BM} also involves
the surreals $\kappa_x\in \fA$ defined by
Kuhlmann and Matusinski~\cite{KM}. This is only needed
for $x=-\alpha$, and it follows from the results in \cite{KM} that 
$\kappa_{-\alpha}=\omega^{\omega^{-\omega\alpha}}$, where $\omega\alpha$ is the 
usual ordinal product.
Thus in view of Lemma~\ref{lambda2}:

\begin{cor}\label{lambda3} We have $\kappa_{-\alpha}\ =\ \lambda_{-\omega\alpha}\ =\ \omega^{\omega^{-\omega\alpha}}\ =\ \log_{\omega\alpha}\omega$.
\end{cor} 

\noindent
We also use the binary relations $\preceq^K$, $\succ^K$, and $\asymp^K$ on $\No^{>\N}$ defined by
\begin{align*} x\ \preceq^K\ y&\quad\Longleftrightarrow\quad  \text{$x \le \exp_n(y)$  for some $n$,}\\
x\ \succ^K\ y &\quad\Longleftrightarrow\quad \text{$x > \exp_n(y)$  for all $n$,}\\
x\ \asymp^K\ y &\quad\Longleftrightarrow\quad \text{$x \preceq^K y$ and $y \preceq^K x$.}
\end{align*} 
We refer to \cite[5.3]{BM} for proofs of some basic facts about these relations and the $\kappa_x$ such as: $\asymp^K$ is an equivalence relation on $\No^{>\N}$ with convex equivalence classes, every $\asymp^K$-equivalence class has a unique element $\kappa_x$ in it, and this element is the simplest element of this equivalence class. Also,
$\kappa_x\le_s \kappa_y$ iff $x\le_s y$.

\subsection*{Defining the pre-derivation for $\der_{\BM}$} The pre-derivation $D$ with $\der_D=\der_{\BM}$ is denoted by $\der_{\mathbb{L}}$ in
\cite[Definition 6.7]{BM}, and by $\der_{\fA}$ in this paper. It is given by
$$\der_{\fA}(\lambda)\ :=\ \prod_n\log_n \lambda\bigg/\prod_{\alpha} \log_{\alpha}\omega$$
with $\alpha$ in the denominator ranging over the ordinals
such that $\log_{\alpha} \omega\ge \log_n\lambda$ for some $n$; to facilitate comparison with \cite{BM} we note that this condition on $\alpha$ is equivalent to $\lambda \preceq^K \log_{\alpha}\omega$. (The products on the right exist, since 
$\log_n\lambda$ and $\log_{\alpha}\omega $ are strictly decreasing as functions of $n$ and $\alpha$, respectively.) The above defining identity for $\der_{\fA}$ simplifies the expression in \cite{BM} by our use of infinite products
(instead of exponentials of infinite sums), and of Lemma~\ref{lambda2} and
Corollary~\ref{lambda3} (to get rid of $\kappa$-numbers). 
As  \cite[Section 9]{BM} shows,  $\der_{\fA}$ is in a certain technical sense  
the {\em simplest\/} pre-derivation.

If $\lambda> \exp_n\omega$
for all $n$, then $\der_{\fA}(\lambda)=\prod_n \log_n \lambda$. Another special case is 
$\der_{\fA}(\log_{\alpha}\omega) = 1\big/\prod_{\beta< \alpha} \log_{\beta} \omega$, in particular, $\der_{\fA}(\omega)=1$. For $\epsilon$-numbers we get the following 
(not needed later, but included as an example): 

\begin{lemma}\label{logeps} We have $\log_n\epsilon=\omega^{\omega^{\epsilon - n}}$. Hence $\epsilon\in \fA$ and $$\der_{\fA}(\epsilon)\  =\ \omega^{\omega^\epsilon + \omega^{\epsilon-1}+ \omega^{\epsilon -2} + \cdots}\ =\ \omega^{\epsilon/(1-\omega^{-1})}.$$
\end{lemma}
\begin{proof} From \cite[pp. 179, 180]{G} we get that if $b$, as a 
sequence of pluses and minuses, equals 
$\epsilon$ followed by $\epsilon\omega n$ minuses, with $n\ge 1$ and 
$\epsilon \omega n$ being the ordinal product,
then $b=\omega^{\epsilon-n}$, and $g(b)=\epsilon-(n-1)$. In other words,
$$ g\big(\omega^{\epsilon-n}\big)\ =\ \epsilon-(n-1)\ \qquad(n\ge 1).$$ 
 Using this we prove the lemma by induction on $n$. The case $n=0$ is clear. 
Assume inductively that $\log_n\epsilon=\omega^{\omega^{\epsilon - n}}$. 
Since $g\big(\omega^{\epsilon-(n+1)}\big)=\epsilon-n$, this gives 
$$\exp\big(\omega^{\omega^{\epsilon-(n+1)}}\big)\ =\ \omega^{\omega^{\epsilon - n}},$$
from which we get $\log_{n+1}\epsilon=\omega^{\omega^{\epsilon - (n+1)}}$, as desired. 
\end{proof}

\section{Exhibiting $\No$ as a Suitable Directed Union}\label{sec:dirun}

\noindent
At the end of Section~\ref{sec:pre} we explained how proving
$\T \equiv \No$ (as differential fields) reduces to representing
$\No$ as a directed union of spherically complete grounded $H$-subfields. In this section we obtain such a representation. The reader should beware of considering $\No$ itself as spherically complete, even though the Conway normal form is sometimes summarized as ``$\No=\R\(( \omega^{\No}\)) $''. This is misleading, however, since it suggests that a series like $\sum_\alpha \omega^{-\alpha}$,
where the sum is over all ordinals $\alpha$, is a surreal number. It might perhaps be viewed as a surreal number in a
strictly larger set-theoretic universe, but not in the one we are (tacitly) working in. A better way of understanding
$\No$ as a valued field is as the directed union $\bigcup_{\Gamma} \R[[\omega^\Gamma]]$ with~$\Gamma$ ranging over the subsets of $\No$ that
underly an additive subgroup of $\No$; for example, any $\alpha$ gives $\No(\omega^\alpha)$ as such a $\Gamma$.  
For any such $\Gamma$ the corresponding~$\R[[\omega^\Gamma]]$ is indeed a spherically complete valued subfield of $\No$, but in general $\R[[\omega^\Gamma]]$ is
not closed under
$\der_{\BM}$, and even if it is, it might not be grounded.

In this section we show that for $S=\No(\epsilon)\cup\{-\epsilon\}$, with $\epsilon$ any $\epsilon$-number, the Hahn subgroup
$\Gamma=\R[[\omega^S]]$ of $\No$ gives rise to a spherically complete valued subfield $\R[[\omega^\Gamma]]$ that is closed under $\der_{\BM}$ and grounded as an $H$-subfield of $\No$.

\subsection*{A length bound for $h$} This very useful bound is as follows:

\begin{lemma}\label{3} $l\big(h(y)\big)\ \le\ \omega^{l(y)+1}$.
\end{lemma}
\begin{proof} By \cite[p.~172]{G} the canonical representation
$y=\{y'\}|\{y''\}$ yields
$$h(y)\ =\ \big\{0, h(y')\big\}\big|\big\{h(y''), \omega^y/2^n\big\}.$$
We can assume inductively that the lemma holds for the $y'$ and $y''$ instead of $y$,  
and thus
$l\big(h(y')\big) \le\ \omega^{l(y')+1} < \omega^{l(y)+1}$ for all $y'$, and 
likewise with $y''$ instead of $y'$.
% and thus $l\big(h(y')\big)< \omega^{l(y)+1}$ for all $y'$, and 
%likewise with $y''$ instead of $y'$. 
Also, $l(\omega^y/2^n) \le l(\omega^y)l(1/2^n)< \omega^{l(y)}\omega=\omega^{l(y)+1}$, using \cite[Lemmas 3.6 and 4.1]{DE}. Now appeal to \cite[Theorem 2.3]{G}.   
\end{proof}

\noindent
Recall from Section~\ref{sec:pre} that $h(-\alpha)=\omega^{-(\alpha+1)}$, and so $h(0)=\omega^{-1}$ shows that for $y=0$ the upper bound in Lemma~\ref{3} is attained.

\subsection*{Some spherically complete initial subfields of $\No$} In this subsection we fix an initial subset $I$ of $\No$. Then $\Gamma:=\R[[\omega^I]]$ is an initial additive subgroup of $\No$ by the proof of Theorem~18 in \cite{Eh}. (That theorem considers Hahn fields rather than the Hahn group $\Gamma$, but the same ideas work; we stress that it is the proof of that theorem rather than its statement that matters here.) Moreover,
as Philip Ehrlich mentioned to one of us: 

\begin{lemma} Suppose $I$ has a least element $a$. Then $a=-\alpha$ for some $\alpha$, and $\Gamma$ has a least nontrivial 
archimedean class 
represented by $\omega^{a}$. 
\end{lemma}
\begin{proof} Taking the longest initial segment of $a$ consisting of minus signs
we get the largest ordinal $\alpha$ with $-\alpha\le_s a$. Then $-\alpha\in I$
and $-\alpha\le a$, so $-\alpha=a$. 
\end{proof} 

\noindent
Since $\Gamma$ is initial and an ordered additive group it leads to the initial subfield $K:=\R[[\omega^\Gamma]]$ of $\No$. Note that $K$ is
spherically complete, and if $(a_i)$ is a family in~$K$ for which $\sum_i a_i$ exists, then $\sum_i a_i\in K$. Now $\Gamma=\R[[\omega^I]]$ is also closed
under infinite sums, so if $(\frak{m}_i)$ is a family in $\frak{M}\cap K$ 
such that $\prod_i \frak{m}_i$ exists,
 then $\prod_i \frak{m}_i\in K$. Thus $K$ is closed under infinite sums, and also under infinite products
 of monomials.
This is very useful in showing that for suitable choices of $I$ the field~$K$ is closed under certain surreal derivations. Note however, that if $I$ has a least element, then $K^{>\N}$ is
not closed under $\log$: if $-\alpha$ is the least element of $I$, then $\log_{\alpha}\omega=\omega^{\omega^{-\alpha}}\in K$, but
$\log_{\alpha+1}\omega\notin K$, as $-(\alpha+1)\notin I$.

In order to discuss examples we set  $a^r:=\exp(r\log a)$ for $a>0$ and $r\in \R$, and note agreement with the previously defined $\omega^r$ when $a=\omega$. 
Moreover,
$$(\log_{\alpha}\omega)^r\ =\ \omega^{r\omega^{-\alpha}} \qquad(r\in \R),$$
by the definition of $a^r$, using also $g\big(\omega^{-(\alpha+1)}\big)=-\alpha$ and \cite[Theorem 10.13]{G}.

\begin{examples} For $I=\{0\}$ we get $\Gamma=\R$ and 
$K=\R[[\omega^\R]]$; note that $K$ is closed under~$\der_{\BM}$, but $\omega\in K$ and $\log \omega=\omega^{1/\omega}\notin K$.

For $I=\{0,-1\}$ we have $\Gamma=\R + \R\omega^{-1}$, so $\omega^\Gamma=\omega^{\R}(\log \omega)^{\R}$, and thus $K=\R[[\omega^{\R}(\log \omega)^{\R}]]$, which
is again closed under $\der_{\BM}$. 

Let $I=\{\alpha:\ \alpha\le \epsilon\}$.
Then $\epsilon=\omega^{\omega^\epsilon}\in K$, but Lemma~\ref{logeps} gives $\log \epsilon\notin K$, since $\epsilon -1\notin I$ and so
$\omega^{\epsilon-1}\notin \Gamma$. Likewise we get $\der_{\BM}(\epsilon)\notin K$. 
\end{examples}

\begin{lemma}\label{3a} If $I=\big\{a:\ l(a)< \alpha\big\}$ or $I=\big\{a:\ l(a) \le \alpha\big\}$, then $I\subseteq \Gamma \subseteq K$.
\end{lemma}
\begin{proof} Suppose $I=\big\{a:\ l(a) < \alpha\big\}$. 
(The case $I=\big\{a:\ l(a) \le \alpha\big\}$ is handled in the same way.) Let $a\in I$. Then $a=\sum_x a_x\omega^x$, and if $x\in E(a)$, then $l(x)\le l(\omega^x)\le l(a) < \alpha$ by \cite[Lemmas~3.4,~4.1, and~4.2]{DE}, so $x\in I$. 
Thus $a\in \Gamma$.
This proves $I\subseteq \Gamma$.
Next, if $b\in \Gamma$, then $b=\sum_{x\in I}b_x\omega^x$, and 
so $b\in K$ in view of $I\subseteq \Gamma$. 
\end{proof}

\noindent
The next lemma will also be crucial:

\begin{lemma}\label{2} Suppose $h(I)\subseteq \Gamma$. Then $\log K^{>}\subseteq K$ and for each $a\in K$ and term~$t$ of $a$ we have: $t$ and all terms of $\ell(t)$
lie in $K$. 
\end{lemma}
\begin{proof} Let $a\in K^{>}$ have leading monomial
$\frak{m}=\omega^b$ with $b=\sum_{y\in I}b_y\omega^y$; to get
$\log a\in K$, it is enough that $\log \frak m\in K$; the latter holds because $\log \frak{m}=\sum_y b_y\omega^{h(y)}$. 
This proves $\log K^{>}\subseteq K$. 

Next, let $a\in K$ and let $t$ be a term of $a$; we have to show  that $t$ and all terms of $\ell(t)$ lie in $K$.
As $K\supseteq \R$ is initial, it does contain the term $t$ of its element $a$. We have $t=r\omega^b$ with $r\in \R^\times$
and $b\in \Gamma$, so $b=\sum_{y\in I} b_y\omega^y$, and
thus $\omega^b=\exp\big(\sum_{y\in I}b_y\omega^{h(y)}\big)$.
Hence 
$\ell(t)=\ell(r\omega^b)=\sum_{y\in I} b_y \omega^{h(y)}$
and each of its terms $b_y\omega^{h(y)}$ lies obviously in $K$.
\end{proof}

\begin{cor}\label{cor:2} If $h(I)\subseteq \Gamma$ and $D$ is a pre-derivation with $D(K\cap \fA)\subseteq K$, then $\der_D(K)\subseteq K$.
\end{cor}
\begin{proof} Use the definition of $\der_D$ from Section~\ref{sec:pre}, the fact that $K$ is closed under infinite sums, and
Lemma~\ref{2}.
\end{proof}

\begin{cor}\label{cor:3} Suppose $h(I)\subseteq \Gamma$. Then $\der_{\BM}(K)\subseteq K$.
\end{cor} 
\begin{proof} Let $\lambda\in K\cap \fA$;  by Corollary~\ref{cor:2} we just need to get $\der_{\fA}(\lambda)\in K$. Since $K$ is closed under infinite products, it is enough for this to get
$\log_n \lambda\in K$ for all $n$ (which is the case by Lemma~\ref{2}), and $\lambda_{-\alpha}\in K$ for all $\alpha$ such that $\lambda\preceq^K \lambda_{-\alpha}$.  Given such $\alpha$,
take $n$ with $\log_n\lambda < \lambda_{-\alpha}$. Then
$\lambda_{-\alpha}\le_s\lambda_{-(\alpha +1)}\le_s \log_n\lambda\in K$ by Lemma~\ref{lambda2a},
  and so $\lambda_{-\alpha}\in K$ because
 $K$ is initial. 
\end{proof}

\noindent
It can happen that $h(I)\not\subseteq \Gamma$ and that $K$ is nevertheless closed under $\der_{\BM}$. 
The next lemma gives a useful criterion for that. To
see why that lemma holds, consider a surreal derivation $\der$,
and note that
from $\omega^{\omega^y}=\exp(\omega^{h(y)})$ we get 
$$\der\big(\omega^{\omega^y}\big)\ =\ \omega^{\omega^y}\cdot\der(\omega^{h(y)}),$$
so for any monomial 
$\frak{m}=\omega^b\in K$ we have $b=\sum_{y\in I}b_y\omega^y$, and thus $$\frak{m}\ =\ \exp\left(\sum_{y\in I}b_y\omega^{h(y)}\right),\qquad
\der(\frak{m})\ =\ \frak{m}\cdot
\sum_{y\in I}b_y\der\big(\omega^{h(y)}\big).$$
This leads to:

\begin{lemma}\label{7} Given a surreal derivation $\der$,
the following are equivalent: 
\begin{enumerate}
\item $K$ is closed under $\der$;
\item $\der(\omega^{\omega^y})\in K$ for all $y\in I$;
\item $\der(\omega^{h(y)})\in K$ for all $y\in I$.
\end{enumerate}
\end{lemma}

\subsection*{The surreal fields $K_{\epsilon}$} Given the 
$\epsilon$-number $\epsilon$, we have the initial set
$I:=\No(\epsilon)$, with the corresponding $\Gamma:=\R[[\omega^I]]$ and $K:= \R[[\omega^{\Gamma}]]$. 
In view of Lemmas~\ref{3} and~\ref{3a} we have $h(I)\subseteq I\subseteq \Gamma$, so $\der_{\BM}(K)\subseteq K$ by Corollary~\ref{cor:3}. Thus $K$ is a spherically complete initial $H$-subfield of {\bf No}. However, $I$ has no least element, so~$K$ is not grounded. We repair this by just augmenting $I$ by $-\epsilon$: set
$I_{\epsilon}:= I \cup \{-\epsilon\}$. Then $I_{\epsilon}$ is still
initial, with least element $-\epsilon$, and so we have the
corresponding $\Gamma_{\epsilon}:=\R[[\omega^{I_{\epsilon}}]]$ and 
$K_{\epsilon}:=\R[[\omega^{\Gamma_{\epsilon}}]]$. To get 
$\der_{\BM}(K_{\epsilon}) \subseteq K_{\epsilon}$ we note that $K\subseteq K_{\epsilon}$, and so it suffices by Lemma~\ref{7} that $\der_{\fA}(\omega^{\omega^{-\epsilon}})\in K_{\epsilon}$.
But $\omega^{\omega^{-\epsilon}}=\log_\epsilon \omega$, and 
$$\der_{\fA}(\log_{\epsilon}\omega)\ =\ 1\bigg/\prod_{\alpha< \epsilon} \log_{\alpha} \omega,$$
which lies in $K$, and hence in $K_{\epsilon}$.
Thus $K_{\epsilon}$ is a grounded $H$-subfield of {\bf No}, and
$$ \No\ =\ \bigcup_{\epsilon}K_{\epsilon}.$$ 
Note that Corollary~\ref{cor:3} does not apply to $I_{\epsilon}$, since $h(-\epsilon)=\omega^{-(\epsilon + 1)}\notin \Gamma$; this is why we did the less direct construction via
$I=\No(\epsilon)$.

\medskip\noindent
Since $\omega^{-\epsilon}$ represents the smallest archimedean class of $\Gamma_{\epsilon}$, we have 
$$\max \Psi_{K_{\epsilon}}\ =\ v\big((\omega^{\omega^{-\epsilon}})^\dagger\big)\ =\ v\big((\log_{\epsilon}\omega)^\dagger\big)$$ 
by the proof of Lemma~\ref{grarch}. In view of $(\log_{\epsilon}\omega)^\dagger=
(\log_{\epsilon + 1}\omega)'$ and the remarks at the end of 
Section~\ref{sec:pre}, the representation of $\No$ as an 
increasing union 
$\bigcup_{\epsilon} K_{\epsilon}$ of spherically complete grounded $H$-subfields now
gives $\der_{\BM}(\No)=\No$.
(The proof of $\der_{\BM}(\No)=\No$ in \cite[Section 7]{BM} is different.) 
Thus by the results stated at the end of Section~\ref{sec:pre} we conclude that $\No \equiv \T$, as differential fields.

\section{The Case of Restricted Length} \label{sec:countable}

\noindent
A set $S\subseteq \No$ is said to be of {\em countable type}\/ if
$l(a)$ is countable for all $a\in S$, and all well-ordered subsets
of $S$ as well as all reverse well-ordered subsets of $S$ are countable. (Note that $l(a)$ is countable for every
$a\in \No(\omega_1)$, but that $\No(\omega_1)$ is not of countable type, since it has the set of countable ordinals as an uncountable well-ordered subset.)

\begin{prop}\label{c1} Suppose the subset $S$ of $\No$ is of countable type. Then the additive subgroup $\R[[\omega^{S}]]$ of $\No$ is also of countable type.
\end{prop}
\begin{proof} The case $\alpha=1$ of Esterle~\cite[Lemme~2.2]{E} and the remarks following it yield that
every well-ordered subset of $\R[[\omega^S]]$ is countable. Hence every reverse well-ordered subset of $\R[[\omega^S]]$ is countable as well.  Let
$a\in \R[[\omega^S]]$. Then $a=\sum_{s\in E(a)} a_{s}\omega^s$. Now
$E(a)\subseteq S$ is countable, so the well-ordered set $-E(a)$
has order type
$\mu< \omega_1$. Since $\omega_1$ is regular, we have a
countable ordinal
$\nu$ such that $l(s)\le \nu$ for all $s\in E(a)$.
Then $l(\omega^s)\le \omega^\nu$ for all $s\in E(a)$ by \cite[Lemma~4.1]{DE}, hence $l(a_s\omega^s)\le \omega^{\nu+1}$ for all $s\in E(a)$ by \cite[Proposition~3.6]{DE}. Thus 
$$l(a)\ \le\ \mu\cdot \omega^{\nu +1}\ <\ \omega_1,$$
by \cite[Theorem~5.12]{G}, or \cite[Lemma~4.2,~(3)]{DE}. 
\end{proof}

\noindent
As an example, consider $S:=\No(\omega)$, the set of of dyadic numbers. Then $S$ is of countable type, and so
$\R[[\omega^S]]$ is of countable type. Nevertheless,
$l\big(\R[[\omega^S]]\big)$ is cofinal in $\omega_1$: given any countable ordinal $\mu$, take an order reversing injective map
$\alpha\mapsto s_{\alpha}\colon \mu \to S$; then
$a:=\sum_{\alpha} \omega^{s_{\alpha}}\in \R[[\omega^S]]$ has $l(a)\ge \mu$, by \cite[p.~63]{G}.

\medskip\noindent
Let $\kappa$ be any infinite cardinal. Esterle~\cite[Lemme~2.2]{E} actually tells us for any set~$S\subseteq \No$:
if all well-ordered subsets and all reverse well-ordered
subsets of~$S$ have size $\le \kappa$, then this remains true for the set $\R[[\omega^S]]\subseteq \No$. The next cardinal~$\kappa^{+}$ is regular, so the arguments in the proof of
Proposition~\ref{c1} go through to give the following, where we call $S\subseteq \No$ {\em of type $\kappa$} if $l(a)\le \kappa$ for all $a\in S$ and all well-ordered subsets of $S$ and all reverse well-ordered subsets of $S$ have size  $\le \kappa$. 

\begin{cor}\label{corc1} If $S\subseteq \No$ is of type $\kappa$, then so is $\R[[\omega^{S}]]$.
\end{cor}

\noindent
Next we show that for countable $\mu$ the set $\No(\mu)$ is of countable type. Every element of $\No(\mu)$ has clearly countable length, for countable $\mu$,
and $\No(\mu)$ is closed under $x\mapsto -x$, so the assertion above reduces to:

\begin{lemma}\label{c2} Suppose the ordinal $\mu$ is countable. Then every well-ordered subset of $\No(\mu)$ is countable.
\end{lemma}

\noindent
This may remind the reader of the well-known property of the ordered set $\R$ that every well-ordered subset of $\R$ is countable. Here is a quick proof using that $\R$ has a countable dense subset $\Q$: given any embedding $\alpha\mapsto r_\alpha$ of an infinite
cardinal $\kappa$ into $\R$, pick for each $\alpha< \kappa$ a rational $q_{\alpha}$ such that $r_{\alpha} < q_{\alpha} < r_{\alpha+1}$; it follows that $\kappa=\aleph_0$. However,
such a countable density argument cannot be used for ordered sets 
$\No(\mu)$ when $\mu$ is a countable limit ordinal $>\omega$:

\begin{lemma} Let $\mu$ be an infinite limit ordinal.
Then the ordered set $\No(\mu)$ is dense without endpoints. If $\mu > \omega$, then there exists a collection of $2^{\aleph_0}$ pairwise disjoint open intervals in $\No(\mu)$, which has therefore no countable dense subset.
\end{lemma} 
\begin{proof} The ordinals $\alpha < \mu$ are cofinal
in this ordered set, and there is no largest such $\alpha$.
For $a < b$ in this ordered set, take $\alpha\le l(a), l(b)$
such that $a|_{\alpha}= b|_{\alpha}$ and $a(\alpha) < b(\alpha)$. If $l(b)>\alpha$, then $b(\alpha) = {+}$, so 
$a < {b-} < b$. If $l(a) > \alpha$, then $a(\alpha)={-}$,
so $a < {a+} < b$. Note that 
${b-}, {a+}\in \No(\mu)$, as $\mu$ is a limit ordinal, 

Next, assume $\mu > \omega$. For each nondyadic $r\in \R\subseteq \No$, we have the surreals~${r-}$ and~${r+}$
of length $\omega+1$, and so we obtain the pairwise disjoint
open intervals $({r-}, {r+})$ in $\No(\mu)$.
\end{proof} 

\begin{proof}[Proof of Lemma~\ref{c2}] For $a\in\No(\mu)$ we define 
$\widehat{a}\colon\mu\to\R$ by
$$\widehat{a}(\alpha)\ =\ \begin{cases}
-1 & \text{if $a(\alpha)={-}$,} \\
0 & \text{if $a(\alpha)=0$,} \\
1 & \text{if $a(\alpha)={+}$,} 
\end{cases}$$
For $S=\{\alpha:\ \alpha<\mu\}$ this yields an order-preserving injective map 
$$a\mapsto \sum_{\alpha<\mu}\widehat{a}(\alpha)\omega^{-\alpha}\ \colon\ \No(\mu)\to \R[[\omega^S]].$$
It remains to appeal to Proposition~\ref{c1}.
\end{proof}

\noindent
Essentially the same argument yields the following generalization:

\begin{cor}\label{corc2} If $\kappa$ is an infinite cardinal and
$\mu$ is an ordinal of cardinality $\le \kappa$, then each well-ordered subset of $\No(\mu)$ has cardinality $\le \kappa$.
\end{cor}

\noindent
Note that for a countable $\epsilon$-number $\epsilon$ the initial set $I_{\epsilon}=\No(\epsilon)\cup \{-\epsilon\}$
is of countable type by Lemma~\ref{c2}, and hence
$\Gamma_{\epsilon}$ and $K_{\epsilon}$ are as well by
Proposition~\ref{c1}. Taking the union over all such countable
$\epsilon$ we obtain the set $\No(\omega_1)$ of all surreals of countable length as an increasing union of spherically complete
grounded $H$-sub\-fields $K_{\epsilon}$ of~$\No$. As in Section~\ref{sec:dirun} and using also the model completeness of $T_{\text{small}}^{\text{nl}}=\operatorname{Th}(\T)$ this yields Theorem~\ref{BM2}. The results above lead moreover to the following generalization:

\begin{cor}\label{corcorc2} Let $\kappa$ be any uncountable cardinal. Then
the subfield $\No(\kappa)$ of $\No$ is closed under
$\der_{\BM}$, and $\No(\kappa)\prec \No$, as ordered differential fields.
\end{cor}
\begin{proof} If $\kappa$ is regular we can argue as for $\omega_1$, using Corollaries~\ref{corc1} and ~\ref{corc2} instead of Proposition~\ref{c1} and
Lemma~\ref{c2}. If $\kappa$ is singular, use that it is the supremum of the uncountable regular cardinals below it.
\end{proof}

\section{Constructing Embeddings}\label{sec:emb}

\noindent
So far we have just worked inside $\No$ and established Theorem~\ref{BM2}. In this section we turn to $\T$ and
prove the embedding results: Theorems~\ref{BM1} and ~\ref{BM3}.

\subsection*{Embedding $\T$ into $\No$} Given a Hahn field $\R[[G]]$ over $\R$ we define a map $F\colon \R[[G]]\to \No$ to be {\em strongly additive\/} if for every summable family $(f_i)$ in $\R[[G]]$ the family $\big(F(f_i)\big)$ is summable in $\No$ and $F\left(\sum_i f_i\right)=\sum_i F(f_i)$.
We refer to \cite[Appendix~A]{ADH} for the construction of $\T$ as an exponential ordered field. In this construction $\T$ is
a subfield of a Hahn field $\R[[G^{\text{LE}}]]$: in fact, 
$G^{\text{LE}}$ is a certain directed union of ordered subgroups $G_m{\downarrow}{}_{n}$, and $\T$ is the corresponding directed union of the Hahn fields 
$\R[[G_m{\downarrow}{}_{n}]]$. A map $F\colon\T\to \No$ is said to be {\it strongly additive}\/ if its restriction to each
$\R[[G_m{\downarrow}{}_{n}]]$ is strongly additive. 

\begin{prop} There is a unique strongly additive embedding
$\iota\colon \T \to \No$ of exponential ordered fields that is the
identity on $\R$ and such that $\iota(x_{\T})=\omega$.
\end{prop} 
\begin{proof} We use the notations from \cite[Appendix~A]{ADH} except that the $x$ there is $x_{\T}$ here. The construction of
$\T$ there begins with the Hahn field $E_0=\R[[x_{\T}^\R]]$, and 
there is clearly a (unique) strongly additive ordered field embedding
$i_0\colon E_0 \to \No$ such that $i_0(r)=r$ and $i_0(x_{\T}^r)=\omega^r$ for all $r\in \R$. 
Moreover, $i_0(\ex^b)=\exp(i_0(b))$ for all $b\in B_0$,
and $\exp(i_0(a))> i_0(E_0)$ for all $a\in A_0^{>}$. 
Assume inductively that we have an extension of $i_0$ to a strongly additive ordered field embedding
$i_m\colon E_m=\R[[G_m]] \to \No$ such that $i_m(\ex^b)=\exp(i_m(b))$ for all $b\in B_m$, and $\exp(i_m(a))> i_m(E_m)$ for all $a\in A_m^{>}$. Then one checks easily that $i_m$ extends (uniquely) to a strongly additive ordered field embedding 
$i_{m+1}\colon E_{m+1} \to \No$ such that $i_{m+1}(\ex^b)=\exp(i_{m+1}(b))$ for all $b\in B_{m+1}$, and $\exp(i_{m+1}(a))> i_{m+1}(E_{m+1})$ for all $a\in A_{m+1}^{>}$. 
Taking a union over all $m$ we obtain an embedding 
$$\iota_0\ :=\ \bigcup_m i_m\ \colon\ \R[[x_{\T}^{\R}]]^{\text{E}}\ =\ \bigcup_m \R[[G_m]] \to \No$$ of ordered exponential fields. 
Replacing in the above $\ell_0=x_{\T}$, $G_m$, $\omega$, by 
$\ell_n=\log_n x_{\T}$, 
$G_m{\downarrow}{}_{n}$, $\log_n \omega$, respectively, we obtain likewise an embedding $$\iota_n\  \colon\ \R[[\ell_n^{\R}]]^{\text{E}}\ =\ \bigcup_m\R[[G_m{\downarrow}{}_{n}]] \to \No$$ of ordered exponential fields with 
$\iota_n(\ell_n)=\log_n \omega$. Each
$\iota_{n+1}$ extends $\iota_n$, so we can take the union
over all $n$ to get an embedding $\iota\colon \T \to \No$ as claimed. The uniqueness holds because the smallest subfield
of $\T$ that contains $\R(x_{\T})$ and is closed under exponentiation, taking logarithms of positive elements, and summation of summable families is $\T$ itself.
\end{proof}

\noindent
Next we apply the model completeness of
the theory of the exponential ordered field of real numbers 
(Wilkie~\cite{W}). By \cite{DMM} and \cite{DE}, respectively, 
the ordered exponential fields $\T$ and $\No$ are models of this theory, and so $\iota\colon \T \to \No$ is an 
elementary embedding of ordered exponential fields.

It is easy to check that $\iota\colon\T \to \No$ is also an embedding of 
ordered differential fields. In view of 
$\T\equiv \No$ (as differential fields), and the model completeness
of $T_{\text{small}}^{\text{nl}}$ mentioned
at the end of Section~\ref{sec:pre} we conclude that
$\iota$ is an elementary embedding of ordered differential fields: Theorem~\ref{BM1}.

Is $\iota$ an elementary embedding of 
{\em ordered differential exponential fields}? We don't know; this is related to the open 
problem from~\cite{ADH} to extend
the model-theoretic results there about $\T$ as a differential field
to $\T$ as a differential exponential field.

 It follows easily from the construction of $\T$ and $\iota$ that all 
surreal derivations $\der$ with $\der(\omega)=1$ agree on $\iota(\T)$. 

\begin{prop} Here are some further properties of the map $\iota$:
\begin{enumerate}
\item $\iota(G^{\operatorname{LE}})=\fM\cap \iota(\T)$;
\item $\iota(\T)$ is truncation closed;
\item $\iota(\T)$ is of countable type; in particular, 
$\iota(\T)\subseteq \No(\omega_1)$.
\end{enumerate}
\end{prop}
\begin{proof} Induction on $m$ gives 
$\iota(G_m)\subseteq \fM$, where we use at the inductive step that
$G_{m+1} =\exp(A_{m})G_m$ and
$\iota(A_{m})\subseteq \mathbb{J}$, the latter being a consequence of
$\iota(G_m)\subseteq \fM$. Likewise, $\iota(G_m{\downarrow}{}_{n})\subseteq \fM$
for all $m, n$, and thus $\iota(G^{\operatorname{LE}})\subseteq \fM$.
Since $\iota$ respects infinite sums of monomials, this yields (1), 
and (2) is then an immediate consequence using also that $\T$ is truncation closed in $\R[[G^{\operatorname{LE}}]]$. As to (3), using the results in 
Section~\ref{sec:countable} one shows by induction on $m$ that $\iota(G_m)$, and likewise each $\iota(G_m{\downarrow}{}_{n})$, has countable type. Hence 
$\iota(G^{\operatorname{LE}})$ has countable type, and so does $\iota(\T)$. 
\end{proof}

\noindent
{\em Question} (Elliot Kaplan): can (2) be improved to $\iota(\T)$ being initial?

\subsection*{Embedding $H$-fields into $\No$} Let $\epsilon$ be an $\epsilon$-number; for example, $\epsilon$ could be any uncountable cardinal. We recall from \cite{DE} 
that $\No(\epsilon)$ is a real closed subfield of~$\No$ containing $\R$. 
We consider $\No(\epsilon)$
as a valued subfield of $\No$ with (divisible) ordered value group 
$v\big(\No(\epsilon)^\times\big)$.
We shall need an easy auxiliary result:

\begin{lemma}\label{lemsat} Let $\kappa$ be a regular uncountable cardinal. Then the underlying ordered sets of $\No(\kappa)$ and 
$v\big(\No(\kappa)^\times\big)$ are $\kappa$-saturated.  
\end{lemma}
\begin{proof} Let $A, B\subseteq \No(\kappa)$ have cardinality $<\kappa$, with $A < B$. The regularity of $\kappa$ yields an ordinal $\alpha < \kappa$
such that $l(A\cup B) <\alpha$. By \cite[Theorem 2.3]{G} this gives a surreal $a$ with $l(a)\le \alpha$ such that $A < a < B$, and then 
$a\in \No(\kappa)$. Thus $\No(\kappa)$ is $\kappa$-saturated
as an ordered set. Next, let $P, Q\subseteq \No(\kappa)^{>}$ have cardinality
$< \kappa$, with $v(P)> v(Q)$. Set $A:= \{np:\ n\ge 1,\ p\in P\}$ and
$B:=\{q/n:\ n\ge 1,\ q\in Q\}$. Then $A<B$, and so the above gives $a\in \No(\kappa)$ with $A< a < B$. Then $v(P)> v(a) > v(Q)$, showing that $v\big(\No(\kappa)^\times\big)$ is $\kappa$-saturated as an ordered set. 
\end{proof}

\noindent
For Theorem~\ref{BM3} we need a sharpening of the model completeness of 
the theory $T^{\text{nl}}$ of $\upo$-free newtonian Liouville closed $H$-fields, 
namely, the quantifier elimination (QE) explained in  
~\cite[Introduction to Chapter 16]{ADH}.
The relevant first-order language for QE has in addition to $\mathcal{L}$ 
extra unary predicate symbols $\operatorname{I}, \Upl, \Upo$, to be interpreted 
in a model $L$ of $T^{\text{nl}}$ as sets $\I(L), \Upl(L), \Upo(L)\subseteq L$ according to their defining axioms:
\begin{align*}
\I(a)\ &\Longleftrightarrow\ a=y' \text{ for some $y\prec 1$  in $L$,}\\
\Upl(a)\ &\Longleftrightarrow\ a=-y^{\dagger\dagger} \text{ for some $y\succ 1$  in $L$,}\\
\Upo(a)\  &\Longleftrightarrow\ 4y''+ay=0 \text{ for some $y\in L^\times$.}
\end{align*}
The sets $\I(L), \Upl(L), \Upo(L)\subseteq L$ are convex; their role with respect to QE is
like that of the set of squares in a real closed field. For more on this, see \cite[Introduction]{ADH}. 
A {\em $\Upl\Upo$-field\/} is a substructure 
$\mathbf{K}=(K, I, \Lambda, \Omega)$ of such an
expanded model $(L,\dots)$ of $T^{\text{nl}}$ for which 
$K$ is an $H$-subfield of $L$. This notion of a $ \Upl\Upo$-field 
is studied in detail in \cite[Section 16.3]{ADH}, from which we
take in particular the fact that any $\upo$-free
$H$-field $K$ has a unique expansion to a $\Upl\Upo$-field 
$\mathbf{K}=(K, I, \Lambda, \Omega)$. The proof below assumes familiarity 
with several other results from  \cite[Section 16.3]{ADH}.

\begin{proof}[Proof of Theorem~\ref{BM3}]
Let $\No_{\Upl\Upo}$ be the expansion of $\No$ to a $\Upl\Upo$-field, and let 
$K$ be any $H$-field
with small derivation and constant field $\R$. In order to embed $K$ 
over~$\R$ into 
$\No$, we first expand $K$ to a $\Upl\Upo$-field 
$\mathbf{K}=(K, I, \Lambda, \Omega)$ with $1\notin I$; this can be done in at least one way, and at most two ways, and $1\notin I$ guarantees that all $\Upl\Upo$-field extensions of $\mathbf{K}$ have small derivation. We claim that $\mathbf{K}$ can be embedded into $\No_{\Upl\Upo}$. The ordered field $\R$ with the trivial derivation is an $H$-field and expands to the
$\Upl\Upo$-field $\mathbf{R}:=\big(\R, \{0\}, (-\infty, 0], (-\infty,0]\big)$. The inclusion of $\R$ into
$K$ and into $\No$ are embeddings of $\mathbf{R}$ into
$\mathbf{K}$ and 
$\No_{\Upl\Upo}$, respectively. By taking $\mathbf{E}:=\mathbf{R}$, our claim reduces therefore to proving the following more general statement:

\medskip\noindent
{\em Claim}. Let $\mathbf{E}\subseteq \mathbf{K}$ be an extension of $\Upl\Upo$-fields with $\R$ as their common constant field, and let $i\colon \mathbf{E}\to \No_{\Upl\Upo}$ be an embedding of $\Upl\Upo$-fields that is the identity on~$\R$. Then $i$ extends to an embedding $\mathbf{K}\to \No_{\Upl\Upo}$ of $\Upl\Upo$-fields.

\medskip\noindent
To prove this we first extend $\mathbf{K}$ to make it $\upo$-free, newtonian, 
and Liouville closed; by \cite[16.4.1 and 14.5.10]{ADH} this can be done 
without changing its constant field. Next we apply \cite[16.4.1]{ADH} again,
but this time to $\mathbf{E}$, to arrange  
%$\mathbf{E}$ by its Newton-Liouville closure inside 
%$\mathbf{K}$, using \cite[16.4.1]{ADH}, we arrange
that $\mathbf{E}$ is $\upo$-free. Take a regular 
uncountable cardinal 
$\kappa> \card(K)$ such that $i(E)\subseteq \No(\kappa)$, where
$E$ is the underlying set of $\mathbf{E}$. By Corollary~\ref{corcorc2} we 
have $\No(\kappa)\prec \No$. In view of 
Lemma~\ref{lemsat} and \cite[16.2.3]{ADH} we can then extend $i$ 
to an embedding $K \to \No(\kappa)$. \end{proof}

\subsection*{Final remarks} Suppose the $H$-field $K$ has small derivation
and constant field $\R$. Then Theorem~\ref{BM3} yields an embedding 
$i\colon K \to \No$ over $\R$. Under some reasonable further conditions, 
like $K$ being $\upo$-free and newtonian, can we take $i$ such that
$i(K)$ is truncation closed, or even initial? The interest of such
a result would depend on how canonical the derivation $\der_{\BM}$ is deemed to
be. As already mentioned at the end of the introduction, we doubt that
$\der_{\BM}$ is optimal: the condition on pre-derivations to take
values in $\R^{>}\fM$ seems too narrow. But even with this restriction one can
construct pre-derivations $D\ne \der{_\fA}$ such that Theorems~\ref{BM1} and~\ref{BM3} 
go through for $\No$ equipped with $\der_D$ instead of with $\der_{\BM}$, 
with only minor changes in the proofs.

\end{document}